\documentclass[12pt]{article}
\usepackage{amsmath,amssymb,amsthm}
\numberwithin{equation}{section}
\usepackage{units}
\usepackage{color}
\usepackage[T1]{fontenc}
\usepackage[utf8]{inputenc}
\usepackage{authblk}
\usepackage{bm}
\usepackage{graphicx}

\usepackage{datetime}

\usepackage[colorlinks=true,
linkcolor=webgreen,
filecolor=webbrown,
citecolor=webgreen]{hyperref}

\definecolor{webgreen}{rgb}{0,.5,0}
\definecolor{webbrown}{rgb}{.6,0,0}

\usepackage[toc,page]{appendix}

\newcommand{\Z}{{\mathbb Z}}

\newtheorem{thm}{Theorem}
\newtheorem{theorem}[thm]{Theorem}
\newtheorem{lemma}{Lemma}

\title{A new Fibonacci identity and its associated summation identities}

\author[]{Kunle Adegoke \\\href{mailto:adegoke00@gmail.com}{\tt adegoke00@gmail.com}}
\affil{Department of Physics and Engineering Physics, \mbox{Obafemi Awolowo University}, 220005 Ile-Ife, Nigeria}

\begin{document}

%\date{September 8, 2018}
%\currenttime
\date{}

\maketitle

\begin{abstract}
\noindent We derive a new Fibonacci identity. This single identity subsumes important known identities such as those of Catalan, Ruggles, Halton and others, as well as standard general identities found in the books by Vajda, Koshy and others. We also derive several binomial and ordinary summation identities arising from this identity; in particular we obtain a generalization of Halton's general Fibonacci identity.

\end{abstract}
%\tableofcontents
% \listoffigures
%\chead{\thepage}
\section{Introduction}
As usual, the Fibonacci numbers, $F_n$, and the Lucas numbers, $L_n$, $n\in\Z$, are defined by:
\begin{equation}
F_0  = 0,\;F_1  = 1,\;F_n  = F_{n - 1}  + F_{n - 2}\; (n\ge 2),\quad F_{ - n}  = ( - 1)^{n-1} F_n
\end{equation}
and
\begin{equation}
L_0  = 2,\;L_1  = 1,\;L_n  = L_{n - 1}  + L_{n - 2}\; (n\ge 2) ,\quad L_{ - n}  = ( - 1)^n L_n\,.
\end{equation}
Both $(F_n)_{n\in\Z}$ and $(L_n)_{n\in\Z}$ are examples of a Fibonacci-like sequence. We define a Fibonacci-like sequence, $(G_n)_{n\in\Z}$, as one having the same recurrence relation as the Fibonacci sequence, but with arbitrary initial terms. Thus, given arbitrary integers $G_0$ and $G_1$, not both zero, we define
\begin{equation}
G_n=G_{n-1}+G_{n-2}\; (n\ge 2)\,;
\end{equation}
and also extend the definition to negative subscripts by writing the recurrence relation as
\begin{equation}\label{eq.ozw7j10}
\quad G_{-n}=G_{-n+2}-G_{-n+1}\,.
\end{equation}
In section \ref{sec.identity}, it will be shown that
\begin{equation}
G_{-n}=(-1)^n(L_nG_0-G_n)\,.
\end{equation}
In this paper, we will derive the following identity involving Fibonacci numbers and Fibonacci-like numbers: 
\[
F_{a-b}G_{n+m}=F_{m-b}G_{n+a}+(-1)^{a+b+1}F_{m-a}G_{n+b}\,,
\]
valid for all integers $a$, $b$, $n$ and $m$.

\medskip

Various summation identities emanating from this identity will be derived. In particular, we will derive (section \ref{sec.summation}, identity \eqref{eq.kh2azr9})the following generalization of Halton's identity (see Halton \cite[identity (23)]{halton65}):
\[
\sum_{j = 0}^k {( - 1)^j \binom kjF_{a - b}^j F_{m + a}^{k - j} G_{n + (a - b)k + (m + b)j} }  = ( - 1)^{(a + b)k} F_{m + b}^k G_n\,.
\]
\section{An identity involving Fibonacci and Fibonacci-like \mbox{numbers}}\label{sec.identity}
\begin{theorem}\label{thm.main}
The following identity holds for arbitrary integers $a$, $b$, $m$ and $n$:
\[
\boxed {F_{a-b}G_{n+m}=F_{m-b}G_{n+a}+(-1)^{a+b+1}F_{m-a}G_{n+b}}\,.
\]

\end{theorem}
\begin{proof}
Since both sequences $(F_n)$ and $(G_n)$ have the same recurrence relation, we choose a basis set in $(F_n)$ and express the numbers from $(G_n)$ in this basis. We write
\begin{equation}\label{eq.nbl33qg}
G_{n+m}=\lambda_1F_{m-b}+\lambda_2F_{m-a}\,,
\end{equation}
where $a$, $b$, $n$ and $m$ are arbitrary integers and the coefficients $\lambda_1$ and $\lambda_2$ are to be determined. Setting $m=a$ and $m=b$, in turn, gives
\begin{equation}\label{eq.jcs53o2}
G_{n + a}  = \lambda_1 F_{a - b},\quad G_{n + b}  = \lambda_2 F_{b - a}\,.
\end{equation}
Multiplying through identity \eqref{eq.nbl33qg} by $F_{a-b}F_{b-a}$ gives
\begin{equation}
F_{a - b} F_{b - a} G_{n + m}  = \underbrace{\lambda _1 F_{a - b}}_{G_{n + a}} F_{b - a} F_{m - b}  + \underbrace{\lambda _2 F_{b - a}}_{G_{n + b}} F_{a - b} F_{m - a}\,.
\end{equation}
Thus, we find
\[
\begin{split}
F_{a - b} F_{b - a} G_{n + m}  &= F_{b - a} F_{m - b} G_{n + a}  + F_{a - b} F_{m - a} G_{n + b}\\
&= F_{b - a} F_{m - b} G_{n + a}  + ( - 1)^{a + b + 1} F_{b - a} F_{m - a} G_{n + b}\,;
\end{split}
\]
so that the identity of the theorem is satisfied identically if $a=b$ and numerically if $a\ne b$.
\end{proof}
Since the left hand side of the identity of Theorem \ref{thm.main} does not change under the interchange of $m$ and $n$ and the interchange of $a$ and $-b$ and $b$ and $-a$, we also have the following identities:
\begin{equation}\label{eq.yvmxj6w}
F_{a - b} G_{n + m}  = F_{m + a} G_{n - b}  + ( - 1)^{a + b + 1} F_{m + b} G_{n - a}\,,
\end{equation}
\begin{equation}\label{eq.e1yqzuf}
F_{a - b} G_{n + m}  = F_{n - b} G_{m + a}  + ( - 1)^{a + b + 1} F_{n - a} G_{m + b}
\end{equation}
and
\begin{equation}\label{eq.h84fbk1}
F_{a-b}G_{n+m}=F_{n+a}G_{m-b}+(-1)^{a+b+1}F_{n+b}G_{m-a}\,.
\end{equation}
If we set $a=0=m$, $b=-n$ in identity \eqref{eq.e1yqzuf} and use the fact that $F_{2n}=F_nL_n$, we have
\[
G_{-n}=(-1)^n(L_nG_0-G_n)\,,
\]
providing a direct access to negative-index Fibonacci-like numbers.

\medskip

The presumably new identity in Theorem \ref{thm.main} includes, as particular cases, most known three-term recurrence relations involving Fibonacci numbers, Lucas numbers and the generalized Fibonacci numbers. We will give a couple of examples to illustrate this.

\medskip

Setting $a=0$ and $b=m-n$ in the identity of Theorem \ref{thm.main} gives
\begin{equation}
F_{n-m}G_{n+m}=F_nG_n+(-1)^{n+m+1}F_mG_m\,,
\end{equation}
which is a generalization of Catalan's identity:
\begin{equation}
F_{n-m}F_{n+m}=F_n^2+(-1)^{n+m+1}F_m^2\,.
\end{equation}
Using $m=0$ and $a=c+b$ in the identity and re-arranging the terms, we find
\begin{equation}
F_{c+b}G_{n+b}-F_bG_{n+b+c}=(-1)^{b+1}F_cG_n\,,
\end{equation}
which is a generalization of Vajda \cite[Formulas (19a) and (19b)]{vajda}.

\medskip

Setting $a=0$ and $b=-m$ in the identity of Theorem \ref{thm.main} gives
\begin{equation}
G_{n+m}+(-1)^mG_{n-m}=L_{m}G_n\,,
\end{equation}
which is Vajda \cite[Formula 10a]{vajda}.

\medskip
Setting $b=0$, $a=k$ and $m=2k$ in the identity of Theorem \ref{thm.main} gives
\begin{equation}
F_{n+2k}=L_{k}F_{n+k}+(-1)^{n+k}F_kF_n\,,
\end{equation}
which is Ruggles' identity \cite{cooper17,ruggles63}.

\medskip

Setting $b=-a$ in the identity of Theorem \ref{thm.main} gives
\begin{equation}
F_{2a}G_{n+m}=F_{m+a}G_{n+a}-F_{m-a}G_{n-a}\,,
\end{equation}
with the special case
\begin{equation}
G_{n+m}=F_{m+1}G_{n+1}-F_{m-1}G_{n-1}\,,
\end{equation}
which is a generalization of the following identity (Halton \cite[Identity (63)]{halton65}, Koshy \cite[Identity (44), page 89]{koshy}):
\begin{equation}
F_{n+m}=F_{m+1}F_{n+1}-F_{m-1}F_{n-1}\,.
\end{equation}
Setting $b=2k$, $a=1$ and $b=2k$, $a=0$, in turn, in the identity of Theorem \ref{thm.main} produces
\begin{equation}
F_{2k - 1} G_{n + m}  = F_{m - 2k} G_{n + 1}  + F_{m - 1} G_{n + 2k}
\end{equation}
and
\begin{equation}\label{eq.m2jtlr3}
F_{2k} G_{n + m}  = F_m G_{n + 2k}  - F_{m - 2k} G_n\,.
\end{equation}
Identity \eqref{eq.m2jtlr3} is a generalization of the following well-known addition formula (Vajda \cite[Formula (8)]{vajda}):
\begin{equation}
G_{n + m}  = F_{m - 1} G_{n}  + F_{m} G_{n + 1}\,.
\end{equation}
Setting $a=n$ and $b=-m$ in the identity of Theorem \ref{thm.main} produces
\begin{equation}
F_{2m} G_{2n}  = F_{n + m} G_{n + m}  - F_{n - m} G_{n - m}\,.
\end{equation}
\section{Summation identities involving Fibonacci and \mbox{Fibonacci-like} numbers}\label{sec.summation}
\subsection{Binomial summation identities}
\begin{lemma}[{\cite[Lemma 3]{adegoke18}}]\label{lem.binomial}
Let $(X_n)$ be any arbitrary sequence. Let $X_n$, $n\in\Z$, satisfy a three-term recurrence relation $hX_n=f_1X_{n-\alpha}+f_2X_{n-\beta}$, where $h$, $f_1$ and $f_2$ are non-vanishing complex functions, not dependent on $n$, and $\alpha$ and $\beta$ are integers. Then,
\begin{equation}\label{eq.fe496kc}
\sum_{j = 0}^k {\binom kjf_2^{k - j} f_1^j X_{n - \beta k + (\beta  - \alpha )j} }  = h^k X_n\,,
\end{equation}
\begin{equation}\label{eq.j7k6a8g}
\sum_{j = 0}^k {( - 1)^j \binom kjf_2^{k - j} h^j X_{n + (\alpha  - \beta )k + \beta j} }  = ( - 1)^k f_1^k X_n
\end{equation}
and
\begin{equation}\label{eq.fnwrzi3}
\sum_{j = 0}^k {( - 1)^j \binom kjf_1^{k - j} h^j X_{n + (\beta  - \alpha )k + \alpha j} }  = ( - 1)^k f_2^k X_n\,,
\end{equation}
for $k$ a non-negative integer.

\end{lemma}
\begin{theorem}\label{thm.tutbc4x}
The following identities hold for positive integer $k$ and arbitrary integers $a$, $b$, $n$, $m$:
\begin{equation}\label{eq.a4qiltd}
\sum_{j = 0}^k {( - 1)^{(a + b + 1)(k - j)} \binom kjF_{m - b}^j F_{m - a}^{k - j} G_{n - (m - b)k + (a - b)j} }  = F_{a - b}^k G_n\,,
\end{equation}
\begin{equation}
\sum_{j = 0}^k {( - 1)^{(a + b)j} \binom kjF_{a - b}^j F_{m - a}^{k - j} G_{n - (a - b)k + (m - b)j} }  = ( - 1)^{(a + b)k} F_{m - b}^k G_n\,,
\end{equation}
\begin{equation}\label{eq.kf3kgmr}
\sum_{j = 0}^k {( - 1)^j \binom kjF_{a - b}^j F_{m - b}^{k - j} G_{n + (a - b)k + (m - a)j} }  = ( - 1)^{(a + b)k} F_{m - a}^k G_n\,,
\end{equation}
\begin{equation}\label{eq.sa53jkd}
\sum_{j = 0}^k {( - 1)^{(a + b + 1)(k - j)} \binom kjF_{m + a}^j F_{m + b}^{k - j} G_{n - (m + a)k + (a - b)j} }  = F_{a - b}^k G_n\,,
\end{equation}
\begin{equation}
\sum_{j = 0}^k {( - 1)^{(a+b)j} \binom kjF_{a - b}^j F_{m + b}^{k - j} G_{n - (a - b)k + (m + a)j} }  = ( - 1)^{(a + b)k} F_{m + a}^k G_n
\end{equation}
and
\begin{equation}\label{eq.kh2azr9}
\sum_{j = 0}^k {( - 1)^j \binom kjF_{a - b}^j F_{m + a}^{k - j} G_{n + (a - b)k + (m + b)j} }  = ( - 1)^{(a + b)k} F_{m + b}^k G_n\,.
\end{equation}

\end{theorem}
\begin{proof}
To derive identities \eqref{eq.a4qiltd} -- \eqref{eq.kf3kgmr}, write the identity of Theorem \ref{thm.main} as
\[
F_{a - b} G_n  = F_{m - b} G_{n - (m - a)}  + ( - 1)^{a + b + 1} F_{m - a} G_{n - (m - b)}\,;
\]
identify $h=F_{a-b}$, $f_1=F_{m-b}$, $f_2=(-1)^{a+b+1}F_{m-a}$, $X_n=G_n$, $\alpha=m-a$, $\beta=m-b$ and use these in Lemma \ref{lem.binomial}. Identities \eqref{eq.sa53jkd} -- \eqref{eq.kh2azr9} are obtained from identities \eqref{eq.a4qiltd} -- \eqref{eq.kf3kgmr} by interchanging $a$ and $-b$ and $b$ and $-a$.
\end{proof}
Particular cases of identities \eqref{eq.a4qiltd} -- \eqref{eq.kh2azr9} are the pure Fibonacci binomial summation identities
\begin{equation}
\sum_{j = 0}^k {( - 1)^{(a + b + 1)(k - j)} \binom kjF_{m - b}^j F_{m - a}^{k - j} F_{n - (m - b)k + (a - b)j} }  = F_{a - b}^k F_n\,,
\end{equation}
\begin{equation}
\sum_{j = 0}^k {( - 1)^{(a + b)j} \binom kjF_{a - b}^j F_{m - a}^{k - j} F_{n - (a - b)k + (m - b)j} }  = ( - 1)^{(a + b)k} F_{m - b}^k F_n\,,
\end{equation}
\begin{equation}
\sum_{j = 0}^k {( - 1)^j \binom kjF_{a - b}^j F_{m - b}^{k - j} F_{n + (a - b)k + (m - a)j} }  = ( - 1)^{(a + b)k} F_{m - a}^k F_n\,,
\end{equation}
\begin{equation}
\sum_{j = 0}^k {( - 1)^{(a + b + 1)(k - j)} \binom kjF_{m + a}^j F_{m + b}^{k - j} F_{n - (m + a)k + (a - b)j} }  = F_{a - b}^k F_n\,,
\end{equation}
\begin{equation}
\sum_{j = 0}^k {( - 1)^{(a+b)j} \binom kjF_{a - b}^j F_{m + b}^{k - j} F_{n - (a - b)k + (m + a)j} }  = ( - 1)^{(a + b)k} F_{m + a}^k F_n
\end{equation}
and
\begin{equation}\label{eq.gz6ate0}
\sum_{j = 0}^k {( - 1)^j \binom kjF_{a - b}^j F_{m + a}^{k - j} F_{n + (a - b)k + (m + b)j} }  = ( - 1)^{(a + b)k} F_{m + b}^k F_n\,;
\end{equation}
and the corresponding identities involving both Fibonacci and Lucas numbers:
\begin{equation}
\sum_{j = 0}^k {( - 1)^{(a + b + 1)(k - j)} \binom kjF_{m - b}^j F_{m - a}^{k - j} L_{n - (m - b)k + (a - b)j} }  = F_{a - b}^k L_n\,,
\end{equation}
\begin{equation}
\sum_{j = 0}^k {( - 1)^{(a + b)j} \binom kjF_{a - b}^j F_{m - a}^{k - j} L_{n - (a - b)k + (m - b)j} }  = ( - 1)^{(a + b)k} F_{m - b}^k L_n\,,
\end{equation}
\begin{equation}
\sum_{j = 0}^k {( - 1)^j \binom kjF_{a - b}^j F_{m - b}^{k - j} L_{n + (a - b)k + (m - a)j} }  = ( - 1)^{(a + b)k} F_{m - a}^k L_n\,,
\end{equation}
\begin{equation}
\sum_{j = 0}^k {( - 1)^{(a + b + 1)(k - j)} \binom kjF_{m + a}^j F_{m + b}^{k - j} L_{n - (m + a)k + (a - b)j} }  = F_{a - b}^k L_n\,,
\end{equation}
\begin{equation}
\sum_{j = 0}^k {( - 1)^{(a+b)j} \binom kjF_{a - b}^j F_{m + b}^{k - j} L_{n - (a - b)k + (m + a)j} }  = ( - 1)^{(a + b)k} F_{m + a}^k L_n
\end{equation}
and
\begin{equation}
\sum_{j = 0}^k {( - 1)^j \binom kjF_{a - b}^j F_{m + a}^{k - j} L_{n + (a - b)k + (m + b)j} }  = ( - 1)^{(a + b)k} F_{m + b}^k L_n\,.
\end{equation}
We remark that Halton's identity \cite[Identity 23]{halton65}, from which he derived a very large number of identities of different kinds, involving the Fibonacci numbers, is a particular case of identity \eqref{eq.gz6ate0}, being an evaluation at $b=0$.
\subsection{Non-binomial summation identities}
\begin{lemma}[{\cite[Lemma 1]{adegoke18}}]\label{lem.u4bqbkc}
Let $(X_n)$ and $(Y_n)$ be any two sequences such that $X_n$ and $Y_n$, $n\in\Z$, are connected by a three-term recurrence relation $hX_n=f_1X_{n-\alpha}+f_2Y_{n-\beta}$, where $h$, $f_1$ and $f_2$ are arbitrary non-vanishing complex functions, not dependent on $n$, and $\alpha$ and $\beta$ are integers. Then, the following identity holds for integer $k$:
\[
f_2 \sum_{j = 0}^k {f_1^{k - j} h^j Y_{n - k\alpha  - \beta  + \alpha j} }  = h^{k + 1} X_n  - f_1^{k + 1} X_{n - (k + 1)\alpha }\,. 
\]

\end{lemma}
\begin{theorem}\label{thm.c2gyoa1}
The following identities hold for $a$, $b$, $m$, $n$ and $k$ arbitrary integers:
\begin{equation}\label{eq.hlcov46}
\begin{split}
&F_{a - b} \sum_{j = 0}^k {( - 1)^{(a + b)j} G_{m + b}^{k - j} G_{m + a}^j G_{n - (a - b)k + m + b + (a - b)j} }\\
&\qquad= ( - 1)^{(a + b)k} F_n G_{m + a}^{k + 1}  + ( - 1)^{a + b + 1} F_{n - (a - b)(k + 1)} G_{m + b}^{k + 1}
\end{split}
\end{equation}
and
\begin{equation}\label{eq.ao5nh45}
\begin{split}
&F_{a - b} \sum_{j = 0}^k {( - 1)^{(a + b)j} G_{m - a}^{k - j} G_{m - b}^j G_{n - (a - b)k + m - a + (a - b)j} }\\
&\qquad= ( - 1)^{(a + b)k} F_n G_{m - b}^{k + 1}  + ( - 1)^{a + b + 1} F_{n - (a - b)(k + 1)} G_{m - a}^{k + 1}\,.
\end{split}
\end{equation}

\end{theorem}
\begin{proof}
To prove identity \eqref{eq.hlcov46}, write identity \eqref{eq.yvmxj6w} as
\begin{equation}
G_{m + a} F_n  = ( - 1)^{a + b} G_{m + b} F_{n - (a - b)}  + F_{a - b} G_{n + m + b}\,,
\end{equation}
identify $h=G_{m+a}$, $f_1=(-1)^{a+b}G_{m+b}$, $f_2=F_{a-b}$, $X_n=F_n$, $Y_n=G_{n+m+b}$, $\alpha=a-b$ and $\beta=0$ and use these in Lemma \ref{lem.u4bqbkc}. Identity \eqref{eq.ao5nh45} is obtained from identity \eqref{eq.hlcov46} through the transformation $a\to -b$, $b\to -a$.
\end{proof}
\begin{lemma}[{\cite[Lemma 2]{adegoke18}}]\label{lem.s9jfs7n}
Let $(X_n)$ be any arbitrary sequence, where $X_n$, $n\in\Z$, satisfies a three-term recurrence relation $hX_n=f_1X_{n-\alpha}+f_2X_{n-\beta}$, where $h$, $f_1$ and $f_2$ are arbitrary non-vanishing complex functions, not dependent on $r$, and $\alpha$ and $\beta$ are integers. Then, the following identities hold for integer $k$:
\begin{equation}\label{eq.mxyb9zk}
f_2 \sum_{j = 0}^k {f_1^{k - j} h^j X_{n - k\alpha  - \beta  + \alpha j} }  = h^{k + 1} X_n  - f_1^{k + 1} X_{n - (k + 1)\alpha }\,,
\end{equation}
\begin{equation}\label{eq.cgldajj}
f_1 \sum_{j = 0}^k {f_2^{k - j} h^j X_{n - k\beta  - \alpha  + \beta j} }  = h^{k + 1} X_n  - f_2^{k + 1} X_{n - (k + 1)\beta }
\end{equation}
and
\begin{equation}\label{eq.n2n4ec3}
h\sum_{j = 0}^k {( - 1)^j f_2^{k - j} f_1 ^j X_{n - (\beta  - \alpha )k + \alpha  + (\beta  - \alpha )j} }  = ( - 1)^k f_1 ^{k + 1} X_n  + f_2^{k + 1} X_{n - (\beta  - \alpha )(k + 1)}\,.
\end{equation}

\end{lemma}
\begin{theorem}\label{thm.ik24j18}
The following identities hold for arbitrary integers $a$, $b$, $n$, $m$ and $k$:
\begin{equation}
\begin{split}
&( - 1)^{a + b + 1} F_{m - a} \sum_{j = 0}^k {F_{m - b} ^{k - j} F_{a - b} ^j G_{n - (m - a)k - (m - b) + (m - a)j} }\\
&\qquad\qquad\qquad\qquad= F_{a - b} ^{k + 1} G_n  - F_{m - b} ^{k + 1} G_{n - (m - a)(k + 1)}\,,
\end{split}
\end{equation}
\begin{equation}
\begin{split}
&F_{m - b} \sum_{j = 0}^k {( - 1)^{(a + b + 1)(k - j)} F_{m - a} ^{k - j} F_{a - b} ^j G_{n - (m - b)k - (m - a) + (m - b)j} }\\
&\qquad\qquad\qquad= F_{a - b} ^{k + 1} G_n  - ( - 1)^{(a + b + 1)(k + 1)} F_{m - a} ^{k + 1} G_{n - (m - b)(k + 1)}\,,
\end{split}
\end{equation}
\begin{equation}
\begin{split}
&F_{a - b} \sum_{j = 0}^k {( - 1)^{(a + b)j} F_{m - a} ^{k - j} F_{m - b} ^j G_{n - (a - b)k + (m - a) + (a - b)j} }\\
&\qquad\qquad= ( - 1)^{(a + b)k} F_{m - b} ^{k + 1} G_n  + ( - 1)^{a + b + 1} F_{m - a} ^{k + 1} G_{n - (a - b)(k + 1)}\,,
\end{split}
\end{equation}
\begin{equation}
\begin{split}
&( - 1)^{a + b + 1} F_{m + b} \sum_{j = 0}^k {F_{m + a} ^{k - j} F_{a - b} ^j G_{n - (m + b)k - (m + a) + (m + b)j} }\qquad\\
&\qquad\qquad\qquad\qquad= F_{a - b} ^{k + 1} G_n  - F_{m + a} ^{k + 1} G_{n - (m + b)(k + 1)}\,,
\end{split}
\end{equation}
\begin{equation}
\begin{split}
&F_{m + a} \sum_{j = 0}^k {( - 1)^{(a + b + 1)(k - j)} F_{m + b} ^{k - j} F_{a - b} ^j G_{n - (m + a)k - (m + b) + (m + a)j} }\\
&\qquad\qquad\qquad= F_{a - b} ^{k + 1} G_n  - ( - 1)^{(a + b + 1)(k + 1)} F_{m + b} ^{k + 1} G_{n - (m + a)(k + 1)}\,,
\end{split}
\end{equation}
and
\begin{equation}
\begin{split}
&F_{a - b} \sum_{j = 0}^k {( - 1)^{(a+b)j} F_{m + b} ^{k - j} F_{m + a} ^j G_{n - (a - b)k + (m + b) + (a - b)j} }\\
&\qquad\qquad= ( - 1)^{(a+b)k} F_{m + a} ^{k + 1} G_n  + ( - 1)^{a + b + 1} F_{m + b} ^{k + 1} G_{n - (a - b)(k + 1)}\,.
\end{split}
\end{equation}

\end{theorem}
\begin{proof}
In Lemma \ref{lem.s9jfs7n}, with $X_n=G_n$, use the $h$, $f_1$, $f_2$, $\alpha$ and $\beta$ obtained in the proof of Theorem \ref{thm.tutbc4x}.
\end{proof}
In particular, we have the pure Fibonacci summation identities
\begin{equation}
\begin{split}
&( - 1)^{a + b + 1} F_{m - a} \sum_{j = 0}^k {F_{m - b} ^{k - j} F_{a - b} ^j F_{n - (m - a)k - (m - b) + (m - a)j} }\\
&\qquad\qquad\qquad\qquad= F_{a - b} ^{k + 1} F_n  - F_{m - b} ^{k + 1} F_{n - (m - a)(k + 1)}\,,
\end{split}
\end{equation}
\begin{equation}
\begin{split}
&F_{m - b} \sum_{j = 0}^k {( - 1)^{(a + b + 1)(k - j)} F_{m - a} ^{k - j} F_{a - b} ^j F_{n - (m - b)k - (m - a) + (m - b)j} }\\
&\qquad\qquad\qquad= F_{a - b} ^{k + 1} F_n  - ( - 1)^{(a + b + 1)(k + 1)} F_{m - a} ^{k + 1} F_{n - (m - b)(k + 1)}\,,
\end{split}
\end{equation}
\begin{equation}
\begin{split}
&F_{a - b} \sum_{j = 0}^k {( - 1)^{(a + b)j} F_{m - a} ^{k - j} F_{m - b} ^j F_{n - (a - b)k + (m - a) + (a - b)j} }\\
&\qquad\qquad= ( - 1)^{(a + b)k} F_{m - b} ^{k + 1} F_n  + ( - 1)^{a + b + 1} F_{m - a} ^{k + 1} F_{n - (a - b)(k + 1)}\,,
\end{split}
\end{equation}
\begin{equation}
\begin{split}
&( - 1)^{a + b + 1} F_{m + b} \sum_{j = 0}^k {F_{m + a} ^{k - j} F_{a - b} ^j F_{n - (m + b)k - (m + a) + (m + b)j} }\qquad\\
&\qquad\qquad\qquad\qquad= F_{a - b} ^{k + 1} F_n  - F_{m + a} ^{k + 1} F_{n - (m + b)(k + 1)}\,,
\end{split}
\end{equation}
\begin{equation}
\begin{split}
&F_{m + a} \sum_{j = 0}^k {( - 1)^{(a + b + 1)(k - j)} F_{m + b} ^{k - j} F_{a - b} ^j F_{n - (m + a)k - (m + b) + (m + a)j} }\\
&\qquad\qquad\qquad= F_{a - b} ^{k + 1} F_n  - ( - 1)^{(a + b + 1)(k + 1)} F_{m + b} ^{k + 1} F_{n - (m + a)(k + 1)}\,,
\end{split}
\end{equation}
and
\begin{equation}
\begin{split}
&F_{a - b} \sum_{j = 0}^k {( - 1)^{(a+b)j} F_{m + b} ^{k - j} F_{m + a} ^j F_{n - (a - b)k + (m + b) + (a - b)j} }\\
&\qquad\qquad= ( - 1)^{(a + b)k} F_{m + a} ^{k + 1} F_n  + ( - 1)^{a + b + 1} F_{m + b} ^{k + 1} F_{n - (a - b)(k + 1)}\,;
\end{split}
\end{equation}
and the corresponding results involving Fibonacci and Lucas numbers:
\begin{equation}
\begin{split}
&( - 1)^{a + b + 1} F_{m - a} \sum_{j = 0}^k {F_{m - b} ^{k - j} F_{a - b} ^j L_{n - (m - a)k - (m - b) + (m - a)j} }\\
&\qquad\qquad\qquad\qquad= F_{a - b} ^{k + 1} L_n  - F_{m - b} ^{k + 1} L_{n - (m - a)(k + 1)}\,,
\end{split}
\end{equation}
\begin{equation}
\begin{split}
&F_{m - b} \sum_{j = 0}^k {( - 1)^{(a + b + 1)(k - j)} F_{m - a} ^{k - j} F_{a - b} ^j L_{n - (m - b)k - (m - a) + (m - b)j} }\\
&\qquad\qquad\qquad= F_{a - b} ^{k + 1} L_n  - ( - 1)^{(a + b + 1)(k + 1)} F_{m - a} ^{k + 1} L_{n - (m - b)(k + 1)}\,,
\end{split}
\end{equation}
\begin{equation}
\begin{split}
&F_{a - b} \sum_{j = 0}^k {( - 1)^{(a + b)j} F_{m - a} ^{k - j} F_{m - b} ^j L_{n - (a - b)k + (m - a) + (a - b)j} }\\
&\qquad\qquad= ( - 1)^{(a + b)k} F_{m - b} ^{k + 1} L_n  + ( - 1)^{a + b + 1} F_{m - a} ^{k + 1} L_{n - (a - b)(k + 1)}\,.
\end{split}
\end{equation}
\begin{equation}
\begin{split}
&( - 1)^{a + b + 1} F_{m + b} \sum_{j = 0}^k {F_{m + a} ^{k - j} F_{a - b} ^j L_{n - (m + b)k - (m + a) + (m + b)j} }\qquad\\
&\qquad\qquad\qquad\qquad= F_{a - b} ^{k + 1} L_n  - F_{m + a} ^{k + 1} L_{n - (m + b)(k + 1)}\,,
\end{split}
\end{equation}
\begin{equation}
\begin{split}
&F_{m + a} \sum_{j = 0}^k {( - 1)^{(a + b + 1)(k - j)} F_{m + b} ^{k - j} F_{a - b} ^j L_{n - (m + a)k - (m + b) + (m + a)j} }\\
&\qquad\qquad\qquad= F_{a - b} ^{k + 1} L_n  - ( - 1)^{(a + b + 1)(k + 1)} F_{m + b} ^{k + 1} L_{n - (m + a)(k + 1)}\,,
\end{split}
\end{equation}
and
\begin{equation}
\begin{split}
&F_{a - b} \sum_{j = 0}^k {( - 1)^{(a + b)j} F_{m + b} ^{k - j} F_{m + a} ^j L_{n - (a - b)k + (m + b) + (a - b)j} }\\
&\qquad\qquad= ( - 1)^{(a + b)k} F_{m + a} ^{k + 1} L_n  + ( - 1)^{a + b + 1} F_{m + b} ^{k + 1} L_{n - (a - b)(k + 1)}\,.
\end{split}
\end{equation}
%\section{Some non-linear properties}
\subsection{Sums involving products of Fibonacci or Fibonacci-like numbers in the denominator of the summand}
\begin{lemma}\label{lem.reciprocal1}
Let $(X_n)$ and $(Y_n)$ be any two sequences such that $X_n$ and $Y_n$, $n\in\Z$, are connected by a three-term recurrence relation $hX_n=f_1X_{n-\alpha}+f_2Y_{n-\beta}$, where $f_1$ and $f_2$ are arbitrary non-vanishing complex functions, not dependent on $n$, and $\alpha$, $\beta$ and $k$ are integers. Then,
\[
X_nX_{n - \alpha (k + 1)} f_2 \sum_{j = 0}^k h^{k-j}f_1^j{\frac{{Y_{n - \beta - \alpha k + \alpha j} }}{{X_{n - \alpha k + \alpha j} X_{n - \alpha - \alpha k + \alpha j} }}}  = h^{k+1}X_n - f_1^{k + 1}X_{n - \alpha (k + 1)}\,. 
\]
\end{lemma}
\begin{theorem}
The following identities hold for values of $a$, $b$, $m$, $n$, $k$ for which the summand is non-singular in the summation interval:
\begin{equation}\label{eq.vd3kfav}
\begin{split}
F_nF_{n - (a - b)(k + 1)}F_{a - b} &\sum_{j = 0}^k {( - 1)^{(a + b)j} \frac{{G_{m + a}^{k - j} G_{m + b}^j G_{n + m + b - (a - b)k + (a - b)j} }}{{F_{n - (a - b)k + (a - b)j} F_{n - a + b - (a - b)k + (a - b)j} }}}\\
&\qquad\qquad= F_nG_{m + a}^{k + 1} - ( - 1)^{(a + b)(k + 1)} F_{n - (a - b)(k + 1)}G_{m + b}^{k + 1}\,,
\end{split}
\end{equation}
\begin{equation}\label{eq.jwgeagg}
\begin{split}
F_nF_{n - (a - b)(k + 1)}F_{a - b} &\sum_{j = 0}^k {( - 1)^{(a + b)j} \frac{{G_{m - b}^{k - j} G_{m - a}^j G_{n + m - a - (a - b)k + (a - b)j} }}{{F_{n - (a - b)k + (a - b)j} F_{n + b - a - (a - b)k + (a - b)j} }}}\\
&\qquad\qquad= F_nG_{m - b}^{k + 1} - ( - 1)^{(a + b)(k + 1)} F_{n - (a - b)(k + 1)}G_{m - a}^{k + 1}\,.
\end{split}
\end{equation}

\end{theorem}
\begin{proof}
In Lemma \ref{lem.reciprocal1}, make the identification $X_n=F_n$ and $Y_n=G_{n+m+b}$ and use the $f_1$, $f_2$, $h$, $\alpha$ and $\beta$ obtained in the proof of Theorem \ref{thm.c2gyoa1}.
\end{proof}
Particular cases of identities \eqref{eq.vd3kfav} and \eqref{eq.jwgeagg} are the following:
\begin{equation}
\begin{split}
F_nF_{n - (a - b)(k + 1)}F_{a - b} &\sum_{j = 0}^k {( - 1)^{(a + b)j} \frac{{F_{m + a}^{k - j} F_{m + b}^j F_{n + m + b - (a - b)k + (a - b)j} }}{{F_{n - (a - b)k + (a - b)j} F_{n - a + b - (a - b)k + (a - b)j} }}}\\
&\qquad\qquad= F_nF_{m + a}^{k + 1} - ( - 1)^{(a + b)(k + 1)} F_{n - (a - b)(k + 1)}F_{m + b}^{k + 1}\,,
\end{split}
\end{equation}
\begin{equation}
\begin{split}
F_nF_{n - (a - b)(k + 1)}F_{a - b} &\sum_{j = 0}^k {( - 1)^{(a + b)j} \frac{{F_{m - b}^{k - j} F_{m - a}^j F_{n + m - a - (a - b)k + (a - b)j} }}{{F_{n - (a - b)k + (a - b)j} F_{n + b - a - (a - b)k + (a - b)j} }}}\\
&\qquad\qquad= F_nF_{m - b}^{k + 1} - ( - 1)^{(a + b)(k + 1)} F_{n - (a - b)(k + 1)}F_{m - a}^{k + 1}\,;
\end{split}
\end{equation}
and
\begin{equation}
\begin{split}
F_nF_{n - (a - b)(k + 1)}F_{a - b} &\sum_{j = 0}^k {( - 1)^{(a + b)j} \frac{{L_{m + a}^{k - j} L_{m + b}^j L_{n + m + b - (a - b)k + (a - b)j} }}{{F_{n - (a - b)k + (a - b)j} F_{n - a + b - (a - b)k + (a - b)j} }}}\\
&\qquad\qquad= F_nL_{m + a}^{k + 1} - ( - 1)^{(a + b)(k + 1)} F_{n - (a - b)(k + 1)}L_{m + b}^{k + 1}\,,
\end{split}
\end{equation}
\begin{equation}
\begin{split}
F_nF_{n - (a - b)(k + 1)}F_{a - b} &\sum_{j = 0}^k {( - 1)^{(a + b)j} \frac{{L_{m - b}^{k - j} L_{m - a}^j L_{n + m - a - (a - b)k + (a - b)j} }}{{F_{n - (a - b)k + (a - b)j} F_{n + b - a - (a - b)k + (a - b)j} }}}\\
&\qquad\qquad= F_nL_{m - b}^{k + 1} - ( - 1)^{(a + b)(k + 1)} F_{n - (a - b)(k + 1)}L_{m - a}^{k + 1}\,.
\end{split}
\end{equation}
\begin{lemma}\label{lem.reciprocal}
Let $(X_n)$ be any arbitrary sequence. Let $X_n$, $n\in\Z$, satisfy a three-term recurrence relation $hX_n=f_1X_{n-\alpha}+f_2X_{n-\beta}$, where $f_1$ and $f_2$ are non-vanishing complex functions, not dependent on $n$, and $\alpha$, $\beta$ and $k$ are integers. Then, the following identities hold for arbitrary integers $n$, $\alpha$, $\beta$ and $k$ for which the summand is not singular in the summation interval:
\begin{equation}\label{eq.q1q4vob}
X_nX_{n - \alpha (k + 1)} f_2 \sum_{j = 0}^k h^{k-j}f_1^j{\frac{{X_{n - \beta - \alpha k + \alpha j} }}{{X_{n - \alpha k + \alpha j} X_{n - \alpha - \alpha k + \alpha j} }}}  = h^{k+1}X_n - f_1^{k + 1}X_{n - \alpha (k + 1)}\,,
\end{equation}
\begin{equation}\label{eq.mg9cdve}
X_nX_{n - \beta (k + 1)} f_1 \sum_{j = 0}^k h^{k-j}f_2^j{\frac{{X_{n - \alpha - \beta k + \beta j} }}{{X_{n - \beta k + \beta j} X_{n - \beta - \beta k + \beta j} }}}  = h^{k+1}X_n - f_2^{k + 1}X_{n - \beta (k + 1)}\,,
\end{equation}
and
\begin{equation}\label{eq.l8b84rc}
\begin{split}
X_nX_{n - (\beta - \alpha)(k + 1)}h&\sum_{j = 0}^k {(-1)^jf_1^{k-j}f_2^j\frac{{X_{n + \alpha - (\beta - \alpha)k + (\beta - \alpha)j} }}{{X_{n - (\beta - \alpha)k + (\beta - \alpha)j} X_{n - \beta + \alpha - (\beta - \alpha)k + (\beta - \alpha)j} }}}\\
&\qquad= f_1^{k+1}X_n + (-1)^kf_2^{k+1}X_{n - (\beta - \alpha)(k + 1)}\,.
\end{split}
\end{equation}

\end{lemma}
\begin{theorem}
The following identities hold for values of $a$, $b$, $m$, $n$, $k$ for which the summand is non-singular in the summation interval:
\begin{equation}
\begin{split}
( - 1)^{a + b + 1} F_{m - a} G_n G_{n - (m - a)(k + 1)} &\sum_{j = 0}^k {\frac{{F_{a - b}^{k - j} F_{m - b}^j G_{n - m + b - (m - a)k + (m - a)j} }}{{G_{n - (m - a)k + (m - a)j} G_{n - (m - a) - (m - a)k + (m - a)j} }}}\\
&\quad= F_{a - b} ^{k + 1} G_n  - F_{m - b}^{k + 1} G_{n - (m - a)(k + 1)}\,,
\end{split}
\end{equation}
\begin{equation}
\begin{split}
F_{m - b} G_n G_{n - (m - b)(k + 1)} &\sum_{j = 0}^k {\frac{( - 1)^{(a + b + 1)j}F_{a - b}^{k - j} { F_{m - a}}^j G_{n - (m - a) - (m - b)k + (m - b)j} }{{G_{n - (m - b)k + (m - b)j} G_{n - (m - b) - (m - b)k + (m - b)j} }}}\\
&\quad= F_{a - b} ^{k + 1} G_n  - ( - 1)^{(a + b + 1)(k+1)}F_{m - a}^{k + 1} G_{n - (m - b)(k + 1)}\,,
\end{split}
\end{equation}
\begin{equation}
\begin{split}
F_{a-b}G_nG_{n - (a - b)(k + 1)}&\sum_{j = 0}^k {\frac{(-1)^{(a+b)j}F_{m - b}^{k-j}F_{m-a}^jG_{n + m - a - (a - b)k + (a - b)j} }{{G_{n - (a - b)k + (a - b)j} G_{n - (a - b) - (a - b)k + (a - b)j} }}}\\
&\qquad= F_{m - b}^{k+1}G_n -(-1)^{(a+b)(k+1)}F_{m-a}^{k+1}G_{n - (a - b)(k + 1)}\,,
\end{split}
\end{equation}
\begin{equation}
\begin{split}
( - 1)^{a + b + 1} F_{m + b} G_n G_{n - (m + b)(k + 1)} &\sum_{j = 0}^k {\frac{{F_{a - b}^{k - j} F_{m + a}^j G_{n - m - a - (m + b)k + (m + b)j} }}{{G_{n - (m + b)k + (m + b)j} G_{n - (m + b) - (m + b)k + (m + b)j} }}}\\
&\quad= F_{a - b} ^{k + 1} G_n  - F_{m + a}^{k + 1} G_{n - (m + b)(k + 1)}\,,
\end{split}
\end{equation}
\begin{equation}
\begin{split}
F_{m + a} G_n G_{n - (m + a)(k + 1)} &\sum_{j = 0}^k {\frac{( - 1)^{(a + b + 1)j}F_{a - b}^{k - j} { F_{m + b}}^j G_{n - (m + b) - (m + a)k + (m + a)j} }{{G_{n - (m + a)k + (m + a)j} G_{n - (m + a) - (m + a)k + (m + a)j} }}}\\
&\quad= F_{a - b} ^{k + 1} G_n  - ( - 1)^{(a + b + 1)(k+1)}F_{m + b}^{k + 1} G_{n - (m + a)(k + 1)}\,,
\end{split}
\end{equation}
and
\begin{equation}
\begin{split}
F_{a - b}G_nG_{n - (a - b)(k + 1)}&\sum_{j = 0}^k {\frac{(-1)^{(a + b)j}F_{m + a}^{k-j}F_{m + b}^jG_{n + m + b - (a - b)k + (a - b)j} }{{G_{n - (a - b)k + (a - b)j} G_{n - (a - b) - (a - b)k + (a - b)j} }}}\\
&\qquad= F_{m + a}^{k+1}G_n -(-1)^{(a + b)(k+1)}F_{m + b}^{k+1}G_{n - (a - b)(k + 1)}\,.
\end{split}
\end{equation}

\end{theorem}
\begin{proof}
In Lemma \ref{lem.reciprocal}, make the identification $X_n=G_n$ and use the $f_1$, $f_2$, $h$, $\alpha$ and $\beta$ obtained in the proof of Theorem \ref{thm.ik24j18}.
\end{proof}
In particular, we have the pure Fibonacci identities:
\begin{equation}
\begin{split}
( - 1)^{a + b + 1} F_{m - a} F_n F_{n - (m - a)(k + 1)} &\sum_{j = 0}^k {\frac{{F_{a - b}^{k - j} F_{m - b}^j F_{n - m + b - (m - a)k + (m - a)j} }}{{F_{n - (m - a)k + (m - a)j} F_{n - (m - a) - (m - a)k + (m - a)j} }}}\\
&\quad= F_{a - b} ^{k + 1} F_n  - F_{m - b}^{k + 1} F_{n - (m - a)(k + 1)}\,,
\end{split}
\end{equation}
\begin{equation}
\begin{split}
F_{m - b} F_n F_{n - (m - b)(k + 1)} &\sum_{j = 0}^k {\frac{( - 1)^{(a + b + 1)j}F_{a - b}^{k - j} { F_{m - a}}^j F_{n - (m - a) - (m - b)k + (m - b)j} }{{F_{n - (m - b)k + (m - b)j} F_{n - (m - b) - (m - b)k + (m - b)j} }}}\\
&\quad= F_{a - b} ^{k + 1} F_n  - ( - 1)^{(a + b + 1)(k+1)}F_{m - a}^{k + 1} F_{n - (m - b)(k + 1)}\,,
\end{split}
\end{equation}
\begin{equation}
\begin{split}
F_{a-b}F_nF_{n - (a - b)(k + 1)}&\sum_{j = 0}^k {\frac{(-1)^{(a+b)j}F_{m - b}^{k-j}F_{m-a}^jF_{n + m - a - (a - b)k + (a - b)j} }{{F_{n - (a - b)k + (a - b)j} F_{n - (a - b) - (a - b)k + (a - b)j} }}}\\
&\qquad= F_{m - b}^{k+1}F_n -(-1)^{(a+b)(k+1)}F_{m-a}^{k+1}F_{n - (a - b)(k + 1)}\,,
\end{split}
\end{equation}
\begin{equation}
\begin{split}
( - 1)^{a + b + 1} F_{m + b} F_n F_{n - (m + b)(k + 1)} &\sum_{j = 0}^k {\frac{{F_{a - b}^{k - j} F_{m + a}^j F_{n - m - a - (m + b)k + (m + b)j} }}{{F_{n - (m + b)k + (m + b)j} F_{n - (m + b) - (m + b)k + (m + b)j} }}}\\
&\quad= F_{a - b} ^{k + 1} F_n  - F_{m + a}^{k + 1} F_{n - (m + b)(k + 1)}\,,
\end{split}
\end{equation}
\begin{equation}
\begin{split}
F_{m + a} F_n F_{n - (m + a)(k + 1)} &\sum_{j = 0}^k {\frac{( - 1)^{(a + b + 1)j}F_{a - b}^{k - j} { F_{m + b}}^j F_{n - (m + b) - (m + a)k + (m + a)j} }{{F_{n - (m + a)k + (m + a)j} F_{n - (m + a) - (m + a)k + (m + a)j} }}}\\
&\quad= F_{a - b} ^{k + 1} F_n  - ( - 1)^{(a + b + 1)(k+1)}F_{m + b}^{k + 1} F_{n - (m + a)(k + 1)}\,,
\end{split}
\end{equation}
and
\begin{equation}
\begin{split}
F_{a - b}F_nF_{n - (a - b)(k + 1)}&\sum_{j = 0}^k {\frac{(-1)^{(a + b)j}F_{m + a}^{k-j}F_{m + b}^jF_{n + m + b - (a - b)k + (a - b)j} }{{F_{n - (a - b)k + (a - b)j} F_{n - (a - b) - (a - b)k + (a - b)j} }}}\\
&\qquad= F_{m + a}^{k+1}F_n -(-1)^{(a + b)(k+1)}F_{m + b}^{k+1}F_{n - (a - b)(k + 1)}\,;
\end{split}
\end{equation}
and the corresponding identities involving Lucas and Fibonacci numbers:
\begin{equation}
\begin{split}
( - 1)^{a + b + 1} F_{m - a} L_n L_{n - (m - a)(k + 1)} &\sum_{j = 0}^k {\frac{{F_{a - b}^{k - j} F_{m - b}^j L_{n - m + b - (m - a)k + (m - a)j} }}{{L_{n - (m - a)k + (m - a)j} L_{n - (m - a) - (m - a)k + (m - a)j} }}}\\
&\quad= F_{a - b} ^{k + 1} L_n  - F_{m - b}^{k + 1} L_{n - (m - a)(k + 1)}\,,
\end{split}
\end{equation}
\begin{equation}
\begin{split}
F_{m - b} L_n L_{n - (m - b)(k + 1)} &\sum_{j = 0}^k {\frac{( - 1)^{(a + b + 1)j}F_{a - b}^{k - j} { F_{m - a}}^j L_{n - (m - a) - (m - b)k + (m - b)j} }{{L_{n - (m - b)k + (m - b)j} L_{n - (m - b) - (m - b)k + (m - b)j} }}}\\
&\quad= F_{a - b} ^{k + 1} L_n  - ( - 1)^{(a + b + 1)(k+1)}F_{m - a}^{k + 1} L_{n - (m - b)(k + 1)}\,,
\end{split}
\end{equation}
\begin{equation}
\begin{split}
F_{a-b}L_nL_{n - (a - b)(k + 1)}&\sum_{j = 0}^k {\frac{(-1)^{(a+b)j}F_{m - b}^{k-j}F_{m-a}^jL_{n + m - a - (a - b)k + (a - b)j} }{{L_{n - (a - b)k + (a - b)j} L_{n - (a - b) - (a - b)k + (a - b)j} }}}\\
&\qquad= F_{m - b}^{k+1}L_n -(-1)^{(a+b)(k+1)}F_{m-a}^{k+1}L_{n - (a - b)(k + 1)}\,,
\end{split}
\end{equation}
\begin{equation}
\begin{split}
( - 1)^{a + b + 1} F_{m + b} L_n L_{n - (m + b)(k + 1)} &\sum_{j = 0}^k {\frac{{F_{a - b}^{k - j} F_{m + a}^j L_{n - m - a - (m + b)k + (m + b)j} }}{{L_{n - (m + b)k + (m + b)j} L_{n - (m + b) - (m + b)k + (m + b)j} }}}\\
&\quad= F_{a - b} ^{k + 1} L_n  - F_{m + a}^{k + 1} L_{n - (m + b)(k + 1)}\,,
\end{split}
\end{equation}
\begin{equation}
\begin{split}
F_{m + a} L_n L_{n - (m + a)(k + 1)} &\sum_{j = 0}^k {\frac{( - 1)^{(a + b + 1)j}F_{a - b}^{k - j} { F_{m + b}}^j L_{n - (m + b) - (m + a)k + (m + a)j} }{{L_{n - (m + a)k + (m + a)j} L_{n - (m + a) - (m + a)k + (m + a)j} }}}\\
&\quad= F_{a - b} ^{k + 1} L_n  - ( - 1)^{(a + b + 1)(k+1)}F_{m + b}^{k + 1} L_{n - (m + a)(k + 1)}\,,
\end{split}
\end{equation}
and
\begin{equation}
\begin{split}
F_{a - b}L_nL_{n - (a - b)(k + 1)}&\sum_{j = 0}^k {\frac{(-1)^{(a + b)j}F_{m + a}^{k-j}F_{m + b}^jL_{n + m + b - (a - b)k + (a - b)j} }{{L_{n - (a - b)k + (a - b)j} L_{n - (a - b) - (a - b)k + (a - b)j} }}}\\
&\qquad= F_{m + a}^{k+1}L_n -(-1)^{(a + b)(k+1)}F_{m + b}^{k+1}L_{n - (a - b)(k + 1)}\,.
\end{split}
\end{equation}

\hrule

\noindent 2010 {\it Mathematics Subject Classification}:
Primary 11B39; Secondary 11B37.

\noindent \emph{Keywords: }
Fibonacci number, Lucas number, Fibonacci-like number.

\hrule

%\noindent Concerned with sequences: 
%A000045, A000129

%\hrule

\end{document}